\documentclass[12pt]{amsbook}
\usepackage[all,cmtip]{xy}

\usepackage{a4wide}
\usepackage{pxfonts}
\usepackage{wasysym}
\usepackage{amssymb,amsxtra,mathrsfs}
\usepackage{eucal}
\usepackage{setspace,upgreek}
\usepackage{mathtools}
\usepackage{url}
\usepackage[hidelinks]{hyperref}

\newcommand{\mf}[1]{\mathfrak{#1}}
\newcommand{\mc}[1]{\mathcal{#1}}

\newcommand{\up}{\upharpoonright}
\newcommand{\w}[1]{\widetilde{#1}}
\newcommand{\ze}{\mathbf{0}} 
\newcommand{\un}{\mathbf{1}} 
\newcommand{\RM}{\mathcal{R}} 
\newcommand{\p}{\mathcal{P}}
\newcommand{\N}{\mathbb{N}}
\newcommand{\K}{\mathbb{K}}
\newcommand{\R}{\mathbb {R}}
\newcommand{\C}{\mathbb {C}}
\newcommand{\A}{\mathcal{A}}
\newcommand{\ov}{\overline}

\newcommand{\ra}{\right\rangle}
\newcommand{\lr}[2]{\left\langle #1,#2\ra}
\newtheorem{theorem}{Theorem}
\newtheorem{lemma}[theorem]{Lemma}

\newtheorem
{corollary}
[theorem]
{Corollary}
\newtheorem{definition}
[theorem]
{Definition}

\newtheorem{notations}
[theorem]
{Notations}

\newtheorem{remark}
[theorem]
{Remark}

\numberwithin{subsection}{section}
\numberwithin{theorem}{section}
\numberwithin{equation}{section}

\begin{document}

\title
[Differentials in Banach Algebras]
{Frechet
Differential
of a 
Power Series
in 
Banach Algebras
}

\author{Benedetto Silvestri}

\email{bndsilvestri@proton.me}

\address{440 W Kelly Ave Apt 6 Jackson, WY 83001-8544 USA (via express couriers signature required)}

\urladdr{\url{https://arxiv.org/a/silvestri_b_1}}

\urladdr{https://orcid.org/0000-0002-8042-4932}

\thanks{
The Author would like to thank
Prof. Victor Burenkov
and 
Prof.
Massimo Lanza de Cristoforis
for their helpful comments 
in reading 
the manuscript.
He also acknowledges
the anonymous referee
for his comments 
which
helped to
improve the presentation
of the work.
Reasearch supported by 
the
Engineering and Physical 
Sciences Research Council
(EPSRC)}

\keywords{
Fr$\Acute{e}$chet 
differentiation
in Banach algebras,
Functional Calculus.
}
\subjclass[2010]{58C20, 46H, 47A60}
\begin{abstract}
We 
present
two
new 
forms
in which
the Frechet differential
of a power series
in a unitary
Banach algebra
can be
expressed
in terms of 
absolutely convergent
series 
involving
the commutant
$C(T):A\mapsto [A,T]$.
Then we apply
the results
to study series of vector-valued
functions on domains in Banach spaces
and to the 
analytic functional
calculus in a 
complex Banach space.
\end{abstract}
\maketitle
\section{Introduction}
In this work
we consider a general 
formula for
the
Fr$\Acute{e}$chet 
differential 
of a power series in 
a unitary Banach algebra
$\A$
over
$\K\in\{\R,\C\}$.
In
Theorem
\ref{05051540c}
it is proved 
that
the Fr$\Acute{e}$chet differential
of the
map
$
g:
\A\ni T\mapsto 
\sum_{n=0}^{\infty}
\alpha_{n}
T^{n}
\in\A
$,
can be expressed
in terms of 
absolutely convergent
series 
involving
the commutant
$
C(T):
\A\ni h
\mapsto
hT-Th
\in\A
$
with $T\in\A$,
in three different forms
containing
$C(T)$,
$C(T^{n})$
or
$C(T)^{n}$.
\par
The two forms 
of the 
Fr$\Acute{e}$chet 
differential 
of a power series
containing
$C(T)$
and
$C(T^{n})$
given in 
statements $(1)$
and $(2)$ of
Theorem
\ref{05051540c}
are new.
While 
we
give a different proof
respect
to 
\cite{Rudin1},
of the  
known 
formula
in statement $(3)$,
containing
the form
$C(T)^{n}$.
\par
These results
are then applied
to study series of vector-valued
functions on domains in Banach spaces
and also applied to analytic functional
calculus in complex Banach spaces.
\par
The
commutant 
$C(T)-$forms in the differential of 
a power series $g$ in a 
noncommutative Banach algebra $\A$,
allows 
to strongly 
simplify the formula
of the derivative
of a function of the type
$\R\supseteq D\ni t\mapsto 
g(\mc{T}(t))$,
whenever
$C(\mc{T}(t))^{n}
\left(
\frac{d\,\mc{T}}{dt}(t)
\right)=0$
for some $n\in\N-\{0\}$,
(see \eqref{16210903} and \eqref{05180801}).
Here 
$\mc{T}$ is a derivable map defined on an open
subset $D$ of $\R$ and with values in $\A$,
and
$D\ni t\mapsto \frac{d\,\mc{T}}{dt}(t)\in\A$
is 
the derivative of $\mc{T}$.
In a similar way we obtain
simplification also for 
the more general case of differential maps
(see Remark \ref{17081112}). 
To obtain
this type of simplification procedures
in calculating the derivative 
or the differential maps
of functions
valued in a noncommutative Banach algebra, 
represents one of the main motivations of this work.
\par
Let us start fixing some notations.
Let
$\mc{A}$
be
a unitary
Banach algebra
over $\K$
and
denote by
$B(\A)$
the unitary Banach algebra
of all bounded linear 
operators on $\mc{A}$
with the standard $\sup-$norm.
Define
the 
linear
maps
$\RM(T):
\A\ni h\mapsto T h\in\A$
and
$\mc{L}(T):
\A\ni h\mapsto h T\in\A$
for all $T\in\A$,
then it results that
$\RM,\mc{L}\in B(\A,B(\A))$
such that
$\|\RM\|_{B(\A,B(\A))}\leq 1$
and
$\|\mc{L}\|_{B(\A,B(\A))}\leq 1$.
Let
$
g(\lambda)
\doteqdot
\sum_{n=0}^{\infty}
\alpha_{n}
\lambda^{n}
$,
with $\lambda\in\K$,
the
coefficients 
$\alpha_{n}$
in
$\K$
and
$R>0$
its
radius of convergence.
In order to simplify the notations,
we convein to denote 
by the same symbol $g$, both
the functions: 
the numerical map
$g(\lambda)=\sum_{n=0}^{\infty}\alpha_{n}\lambda^{n}\in\K$,
with $\lambda\in\K$ such that
$|\lambda|<R$,
and the $\mc{A}-$valued map
$g(T)=\sum_{n=0}^{\infty}\alpha_{n}T^{n}\in\mc{A}$,
with $T\in\mc{A}$ such that
$\|T\|_{\mc{A}}<R$,
where
$\|\cdot\|_{\mc{A}}$ is the norm on $\mc{A}$,
(see Def. \ref{03071132}).
\par
Hence
if we denote
by
$g^{(p)}$
the $p-$derivative 
of the numerical map $g$,
we have
$g^{(p)}(\lambda)=
\sum_{n=p}^{\infty}
p!\dbinom{n}{p}
\alpha_{n}\lambda^{n-p}\in\K$,
with $\lambda\in\K$ such that
$|\lambda|<R$,
while by considering
$g^{(p)}$ as 
a 
$B(\A)-$valued map
we obtain
$g^{(p)}(Q)=
\sum_{n=0}^{\infty}
p!\dbinom{n}{p}\alpha_{n}Q^{n-p}\in B(\A)$,
with 
$Q\in B(\A)$ 
such that
$\|Q\|_{B(\A)}<R$.
Thus
we have
for all $T\in\A$ such that
$\|T\|_{\A}<R$
\begin{equation}
\label{18270701}
g^{(p)}(\RM(T))=
\sum_{n=p}^{\infty}
p!\dbinom{n}{p}\alpha_{n}\RM(T)^{n-p}\in B(\A).
\end{equation}
Denote
by
$B_{r}(\ze)$
a
ball of
radius
$r>0$
in 
$\A$
and
let
$g$ be considered as
an $\A-$valued map, so
$g:B_{R}(\ze)\ni T\mapsto\sum_{n=0}^{\infty}\alpha_{n}T^{n}\in\A$,
then 
$g^{[1]}:B_{R}(\ze)\to B(\A)$
denotes 
the 
Fr\'{e}chet
differential map
of $g$.
Therefore
for $T\in B_{R}(\ze)$
the element
$g^{[1]}(T)\in B(\A)$
is
uniquely determined
by the following
$$
\lim_{
\begin{subarray}{l}
h\to\ze
\\ 
h\neq\ze
\end{subarray}
}
\frac{
\|
g(T+h)
-
g(T)
-
g^{[1]}(T)(h)
\|_{\A}
}
{\|h\|_{\A}}
=0.
$$
Finally given
a series
$N=\sum_{n=0}^{\infty}P_{n}$,
where
$P_{n}:\A\to B(\A)$
for all $n\in\N$,
we say that
it converges 
absolutely uniformly
on $B_{r}(\ze)$,
or 
absolutely uniformly
for 
$T\in B_{r}(\ze)$,
if
$$
\sum_{n=0}^{\infty}
\sup_{T\in B_{r}(\ze)}\|P_{n}(T)\|_{B(\A)}
<\infty.
$$
For a more general definition see Def. \ref{02102006}.
\par
It is a
well-known
result
that
a
power series 
$
g(T)
\doteqdot
\sum_{n=0}^{\infty}
\alpha_{n}
T^{n}
$
in a Banach algebra $\A$
is
Fr\'{e}chet differentiable
term by term,
the corresponding
power
series
of
its
Fr\'{e}chet
differential
$g^{[1]}$
is absolutely
uniformly convergent
on
$B_{r}(\ze)$
in the norm topology
of
$B(\A)$
for
all $0<r<R$,
and finally
that
$g^{[1]}$
is continuous,
where
the radius of convergence
$R$
of 
$
\sum_{n=0}^{\infty}
\alpha_{n}
\lambda^{n}
$
is different to zero.
\par
The 
Fr\'{e}chet
differentiability
of $g$
can be seen as a particular
case
of
the Fr\'{e}chet differentiability
of 
a power
series
of
polynomials
between two
Banach spaces
over $\mathbb{K}$,
whose
proof
for $\mathbb{K}=\C$,
was given
for the first time
in
\cite{martin};
while
the one 
for $\mathbb{K}=\R$,
given for the first time
in \cite{Michal}, 
used
a weak form of Markoff's 
inequality for the derivative 
of a polynomial,
see \cite{Schaeffer}.
\par
Our proof
in Lemma \ref{05051540}
of
the
Fr$\Acute{e}$chet 
differentiability 
term by term
of $g$
has the advantage of
giving 
for the particular
case of
Banach algebras
a unified approach
for both the cases
real and complex.
\par
We are now able to 
state 
the 
results 
of the main
\textbf{Theorem
\ref{05051540c}}
of this work.
We
give
for the first
time the Fr\'{e}chet differential 
$g^{[1]}$
of the $\A-$valued
function
$g(T)=
\sum_{n=0}^{\infty}\alpha_{n}T^{n}$,
in 
a
\emph{
$C(T)-$
depending absolutely
uniformly convergent
series
on $B_{r}(\ze)$,
for all $0<r<R$,
in 
\eqref{07071547i}
and
in a
$C(T^{k})-$
depending absolutely
uniformly convergent
series
on $B_{r}(\ze)$,
for all $0<r<R$
and
with $k\geq 1$,
in 
\eqref{07071636i}}.
This 
allows
us
to give immediately 
a simplified 
formula 
for
the
value
$
g^{[1]}(T)(h)
$
in case of
the commutativity 
$[T,h]=\ze$,
with
$T\in B_{R}(\ze)$
and
$h\in\A$,
(see Remark \ref{17081112}).
\par
Finally
we
give a different proof
respect
to 
\cite{Rudin1}
and in such a way generalizing
that
in 
\cite{BurenkovDiff},
of the  
known 
formula
in 
\eqref{16081549i},
in case
$0<r<\frac{R}{3}$.
\begin{enumerate}
\item
for all 
$
T\in B_{R}(\ze)
$
\begin{equation}
\label{07071547i}
g^{[1]}(T)
=
\sum_{n=1}^{\infty}
n\alpha_{n}\mathcal{L}(T)^{n-1}
-
\left\{
\sum_{p=0}^{\infty}
\left\{
\sum_{n=p+2}^{\infty}
(n-p-1)
\alpha_{n}
\mathcal{L}(T)^{n-(2+p)}
\right\}
\RM(T)^{p}
\right\}
C(T)
\end{equation}
(here
all the series
converge
absolutely
uniformly
on
$B_{r}(\ze)$
for all
$
0<r<R
$),
\item
for all $T\in B_{R}(\ze)$
\begin{equation}
\label{07071636i}
g^{[1]}(T)
=
\sum_{n=1}^{\infty}
n\alpha_{n}\mathcal{L}(T)^{n-1}
-
\sum_{k=2}^{\infty}
\left\{
\sum_{n=k}^{\infty}
\alpha_{n}
\mathcal{L}(T)^{n-k}
\right\}
C(T^{k-1}).
\end{equation}
(here
all
the series
converge 
absolutely
uniformly
on
$B_{r}(\ze)$
for all
$0<r<R$),
\item
for all
$
T\in
B_{\frac{R}{3}}(\ze)
$
\begin{equation}
\label{16081549i}
g^{[1]}(T)
=
\sum_{p=1}^{\infty}
\frac{1}{p!}
g^{(p)}(\RM(T))
C(T)^{p-1}.
\end{equation}
(here
the
series
converges 
absolutely
uniformly
on
$
B_{r}(\ze)
$
for
all
$
0<r<\frac{R}{3}
$,
$g^{(p)}:\K\to\K$
is
the
$p-$th derivative
of the function
$g$
and 
$g^{(p)}(\RM(T))$
is given in \eqref{18270701}).
\par
Finally
we applied these results
in
Corollary \ref{08051458}, 
Remarks. \ref{G08051458} and \ref{14081324},
for 
describing the differential map
of a series of vector-valued functions
differentiable on domains in Banach spaces,
and
in Cor. \ref{09051552}
to study
the differential map of the function
$X\supseteq D\ni x\mapsto g(\mc{T}(x))\in B(G)$.
Here $G$ and $X$ are Banach spaces, 
$D$ is an open set of $X$,
$g$ is the operator-valued map coming by the analytic
functional calculus on $G$ and
$\mc{T}:D\to B(G)$ is a differential map, where 
$B(G)$ is the unitary Banach algebra of all bounded linear operators on $G$.
\end{enumerate}
\section{
Fr\'{e}chet
differential
of
a
power
series
of
differentiable
functions
}
\label{15061005}
\begin{notations}
\label{05051541}
We 
denote 
by
$\N$
the set of all natural numbers
$
\{
0,1,2,...
\}
$.
Let 
$\K\in\{\R,\C\}$
and
$\lr{G}{\|\cdot\|_{G}}$,
or simply
$G$,
be a
Banach space over
$\K$, 
then 
for all
$a\in G$
and
$r>0$
we define the open ball centered in $a$
of radius $r$, to be the following set
$
B_{r}(a)
\doteqdot
\{v\in G\mid\|v-a\|_{G}<r\}
$,
hence its
closure in $G$
is
$
\ov{B}_{r}(a)
\doteqdot
\ov{B_{r}(a)}
=
\{v\in G\mid\|v-a\|_{G}\leq r\}
$.
\par\hspace{12pt}
Let 
$F,G$ 
be
two 
Banach spaces over $\K$, 
briefly
$\K-$Banach spaces,
then 
$
\lr{B(F,G)}{\|\cdot\|_{B(F,G)}}
$,
will 
denote
the $\K-$Banach space 
of all linear continuous mappings
of $F$ to $G$
and
$
\|U\|_{B(F,G)}
\doteqdot
\sup_{\|v\|_{F}\leq 1}
\|U(v)\|_{G}
$,
we also
set
$
\lr{B(G)}{\|\cdot\|_{B(G)}}
\doteqdot
\lr{B(G,G)}{\|\cdot\|_{B(G,G)}}
$.
\par\hspace{12pt}
Let
$\{G_{1},...,G_{n}\}$ 
be
a finite set of $\K-$Banach spaces,
then 
$
\lr{
\prod_{k=1}^{n}G_{k}
}
{
\|\cdot\|_{\prod_{k=1}^{n}G_{k}}
}
$
is the 
Banach space
where
$\prod_{k=1}^{n}G_{k}$
is the product of the vector spaces
$\{G_{1},...,G_{n}\}$ ,
and
$
\|(v_{1},...,v_{n})\|_{\prod_{k=1}^{n}G_{k}}
\doteqdot
\max_{k\in\{1,...,n\}}
\|v_{k}\|_{G_{k}}
$.
\par\hspace{12pt}
If $G_{k}=G$
for all
$k\in\{1,...,n\}$,
then
we will use the following
notation
$
\lr{G^{n}}{\|\cdot\|_{G^{n}}}
\doteqdot
\lr{\prod_{k=1}^{n}G_{k}}
{
\|\cdot\|_{\prod_{k=1}^{n}G_{k}}
}
$.
Let
$\{F_{1},...,F_{n},G\}$ 
be
a finite set of $\K-$Banach spaces,
then 
$
B_{n}(\prod_{k=1}^{n}F_{k};G)
$
is
the $\K-$vector space
of all $n-$multilinear continuous mappings
defined on $\prod_{k=1}^{n}F_{k}$ 
with
values in $G$.
If 
$F_{k}=F$
for all
$k\in\{1,...,n\}$,
then we set
$
B_{n}(F^{n};G)
\doteqdot
B_{n}(\prod_{k=1}^{n}F_{k};G)
$.
\par\hspace{12pt}
In the sequel
we shall
deal with 
Fr\'{e}chet differentiable functions
$$
f:U\subseteq F\to G
$$
defined on an open set
$U$ of a 
$\K-$Banach space 
$F$
and with values in 
a 
$\K-$Banach space 
$G$.
Its 
Fr\'{e}chet
differential function
will be denoted by
$$
f^{[1]}:U\subseteq F\to B(F,G).
$$
Recall
that
a map
$
f:U\subseteq F\to G
$
is Fr\'{e}chet differentiable at
$x_{0}\in U$
if
there exists 
a 
$T\in B(F,G)$
such that
$$
\lim_{
\begin{subarray}{l}
h\to\ze
\\ 
h\neq\ze
\end{subarray}
}
\frac{
\|
f(x_{0}+h)
-
f(x_{0})
-
T(h)
\|_{G}
}
{\|h\|_{F}}
=0.
$$
$T$ is called the 
Fr\'{e}chet
differential of $f$
at $x_{0}$
and is denoted by
$f^{[1]}(x_{0})$.
$f$ is 
Fr\'{e}chet
differentiable 
on $U$
if 
$f$
is
Fr\'{e}chet
differentiable 
at each $x\in U$,
and in this case 
the
map
$
f^{[1]}:
U\to B(F,G)
$
is called
the 
Fr\'{e}chet
differential function
of
$f$.
For the properties of
Fr\'{e}chet differentials
see
Ch $8$
of the Dieudonne book
\cite{Dieud1}.
\par\hspace{12pt}
\vspace{10pt}
Let 
$\A$ 
be an 
associative 
algebra 
over $\K$
(or briefly 
associative 
algebra)
then
the standard
Lie product on $\A$
is
the following map
$$
[\cdot,\cdot]:\A\times\A
\ni
(A,B)
\mapsto
[A,B]
\doteqdot
AB-BA
\in
\A
$$
the commutator 
of $A,B$,
and
for all $T\in\A$
the adjoint linear map of $T$
is so defined
$ad(T):\A\ni h\mapsto [T,h]\in\A$.
We denote
by
$
\A^{\A}
$
the set of all maps
from
$\A$
to
$\A$,
let
$
\RM:\A\to\A^{\A}
$
and
$
\mathcal{L}:\A\to\A^{\A}
$
be 
defined
by
\begin{equation}
\label{16051546}
\begin{cases}
\RM(T):
\A\ni h\mapsto T h\in\A
\\
\mathcal{L}(T):
\A\ni h\mapsto h T\in\A
\end{cases}
\end{equation}
for all $T\in\A$.
We
also
define
the
map
$
C:\A\to\A^{\A}
$
by
$$
C
\doteqdot
-ad
=
\mathcal{L}
-
\RM.
$$
We consider for any
$n\in\N$
the following mapping
$$
u_{n}:\A\ni T\mapsto T^{n}\in\A.
$$
A
Banach algebra 
over $\K$
(or briefly Banach algebra),
see for example
\cite{dal}
or 
\cite{palmer},
is an associative algebra
$\A$
over
$\K$
with a 
norm
$\|\cdot\|$
on it
such that
$\lr{\A}{\|\cdot\|}$
is a Banach space
and for all
$A,B\in\A$
we have
$$
\|A B\|
\leq
\|A\|
\|B\|.
$$
If $\A$
contains
the unit element
then it is called
unitary 
Banach algebra.
We assume for any unitary Banach algebra
with unit $\un$ that $\|\un\|=1$.
\par\hspace{12pt}
It is easy to verify directly 
that for all $T_{1},T_{2}\in\A$
\begin{equation}
\label{11071436}
[\RM(T_{1}),\mathcal{L}(T_{2})]=\ze.
\end{equation}
By recalling  
definition
\eqref{16051546}
we have for all
$T,h\in\A$
that
$
\|
\RM(T)
(h)
\|_{\A}
\leq
\|T\|_{\A}
\|h\|_{\A}
$,
and
$
\|
\mathcal{L}(T)(h)
\|_{\A}
\leq
\|T\|_{\A}
\|h\|_{\A}
$,
hence
\begin{equation}
\label{10501634}
\RM(T),
\mathcal{L}(T)
\in
B(\A)
\end{equation}
with
\begin{equation}
\label{13221607}
\|
\RM(T)
\|_{B(\A)} 
\leq
\|T\|_{\A},
\|
\mathcal{L}(T)
\|_{B(\A)} 
\leq
\|T\|_{\A},
\|
C(T)
\|_{B(\A)}
\leq
2
\|T\|_{\A}.
\end{equation}
Since
$
\mathcal{L}
$
and
$
\RM
$
are linear mappings
we can conclude that
\begin{equation}
\label{16051547}
\begin{cases}
\mathcal{L},
\RM
\in
B(\A,B(\A))
\\
\|\RM\|_{B(\A,B(\A))},
\|\mathcal{L}\|_{B(\A,B(\A))}
\leq
1.
\end{cases}
\end{equation}
\par
Finally
for 
$l,k\in\N$ and
$\{A_{j}\}_{j=0}^{l}\subset\A$
we use the following conventions
$\prod_{j=k}^{l}A_{j}=\un$
and
$\sum_{j=k}^{l}A_{j}=\ze$
for
$l<k$.
\end{notations}
We now present 
a simple
formula
which 
will be used 
later
to decompose
the commutator
$C(T^{n})$
in
terms
of 
$C(T)$.
\begin{lemma}
\label{06051245}
Let 
$\A$ be 
an 
associative
algebra,
then
for all
$
n\in\N$
and
$A_{1},...,A_{n+1},B\in\A$
we have
\begin{equation}
\label{29081729}
\left[
\prod_{k=1}^{n+1}
A_{k},
B
\right]
=
\sum_{s=0}^{n}
\left(
\prod_{k=1}^{s}
A_{k}
\right)
[A_{s+1},B]
\prod_{j=s+2}^{n+1}
A_{j}.
\end{equation}
\end{lemma}
\proof
We shall prove the statement
by induction.
For $n=0$
\eqref{29081729}
is trivial.
Let
\eqref{29081729}
be true 
for
$n-1$,
then
for each
$
A_{1},...,A_{n+1},B\in\A
$
we have 
\begin{alignat*}{2}
\left[
\prod_{k=1}^{n+1}
A_{k},
B
\right]
&
=
\left[
\left(
\prod_{k=1}^{n}
A_{k}
\right)
A_{n+1}
,
B
\right]
\\
\intertext{
and since
$
[A_{1}A_{2},B]
=
A_{1}
[A_{2},B]
+
[A_{1},B]
A_{2}
$
it follows
}
&
=
\left(
\prod_{k=1}^{n}
A_{k}
\right)
\left[
A_{n+1}
,
B
\right]
+
\left[
\prod_{k=1}^{n}
A_{k},B
\right]
A_{n+1}
\\
\intertext{
and
by hypothesis
of the induction
we conclude
}
&
=
\left(
\prod_{k=1}^{n}
A_{k}
\right)
\left[
A_{n+1},B
\right]
+
\sum_{s=0}^{n-1}
\left(
\prod_{k=1}^{s}
A_{k}
\right)
[A_{s+1},B]
\left(
\prod_{j=s+2}^{n}
A_{j}
\right)
A_{n+1}
\\
&
=
\sum_{s=0}^{n}
\left(
\prod_{k=1}^{s}
A_{k}
\right)
[A_{s+1},B]
\prod_{j=s+2}^{n+1}
A_{j}.
\end{alignat*}
\endproof
\begin{corollary}
\label{06051405}
Let 
$\A$ 
be 
a
unitary
associative
algebra,
then
for all
$n\in\N$
and
$T\in\A$
we have 
$$
C(T^{n+1})
=
\sum_{s=0}^{n}
\RM(T)^{s}
C(T)
\mathcal{L}(T)^{n-s}
=
\sum_{s=0}^{n}
\RM(T)^{s}
\mathcal{L}(T)^{n-s}
C(T).
$$
\end{corollary}
\proof
The second
equality
follows
by
Lemma \ref{06051245}
where
$A_{1}=A_{2}=...=A_{n+1}=T$,
the first one
by
the second
and
\eqref{11071436}.
\endproof
The following equality is stated without proof
in
the exercise $19$, $\S 1$, Ch. $1$
of \cite{BourbLie}.
For the sake 
of completeness
we give a proof.
\begin{lemma}
\label{01071957}
Let 
$\A$
be 
a
unitary
associative 
algebra,
then we have for all
$T\in\A$
and 
$n\in\N$
that 
$$
C(T)^{n}
=
\sum_{k=0}^{n}
(-1)^{k}
\dbinom{n}{k}
\RM(T)^{k}
\mathcal{L}(T)^{n-k}.
$$
\end{lemma}
\proof
Since 
by 
\eqref{11071436}
$
\mathcal{L}(T)
$
and
$\RM(T)$
commute
the statement
follows.
\endproof
\begin{lemma}
\label{01072007}
Let 
$\A$ 
be a 
unitary
associative algebra.
Then for all
$T\in\A$
and 
$n\in\N$
we have
\begin{equation}
\label{02071725}
\sum_{p=1}^{n}
\dbinom{n}{p}
\RM(T)^{n-p}
C(T)^{p-1}
=
\sum_{s=1}^{n}
\RM(T)^{n-s}
\mathcal{L}
(T)^{s-1}.
\end{equation}
\end{lemma}
\proof
Since 
$\mathcal{L}=C+\RM$
and 
since
$C(T)$
and
$\RM(T)$
commute
(cf.\eqref{11071436})
we have
\begin{alignat*}{2}
\sum_{p=1}^{n}
\RM(T)^{n-p}\mathcal{L}(T)^{p-1}
&
=
\sum_{p=1}^{n}
\RM(T)^{n-p}
(C(T)+\RM(T))^{p-1}
\notag
\\
&
=
\sum_{p=1}^{n}
\RM(T)^{n-p}
\sum_{k=0}^{p-1}
\dbinom{p-1}{k}
\RM(T)^{p-1-k}
C(T)^{k}
\notag
\\
&
=\sum_{k=0}^{n-1}
\left(\sum_{p=k+1}^{n}
\dbinom{p-1}{k}
\right)
\RM(T)^{n-1-k}
C(T)^{k}
\notag
\\
&
=\sum_{s=1}^{n}
\left(\sum_{p=s}^{n}
\dbinom{p-1}{s-1}
\right)
\RM(T)^{n-s}
C(T)^{s-1}
\notag
\\
&
=\sum_{s=1}^{n}
\dbinom{n}{s}
\RM(T)^{n-s}
C(T)^{s-1}.
\end{alignat*}
\endproof
\begin{definition}
\label{03071132}
Let
$\A$
be a
unitary
Banach algebra,
$
f(\lambda)
=
\sum_{n=0}^{\infty}
\alpha_{n}
\lambda^{n}
$,
where
the coefficients
$\alpha_{n}\in\K$
and
has the radius of convergence $R>0$.
Then for all
$T\in\A$
such that
$\|T\|_{\A}<R$
we can define
$$
f(T)
\doteqdot
\sum_{n=0}^{\infty}
\alpha_{n}
T^{n}
\in
\A.
$$
\end{definition}
It is well-known that the map $u_{n}$
is Fr\'{e}chet differentiable.
For the sake of completeness we give a direct proof
of the Fr\'{e}chet differential function of $u_{n}$ in several
forms which will be used in the sequel.
\begin{lemma}
\label{05051540a}
Let $\A$ be a unitary Banach algebra.
Then
for all
$n\in\N$
the map
$
u_{n}:
\A
\ni
T
\mapsto
T^{n}
\in
\A
$
is Fr\'{e}chet differentiable
and its Fr\'{e}chet differential map
$
u_{n}^{[1]}:
\A\to B(\A)
$
is such that
for all
$T\in\A$
and
$n\in\N$
\begin{alignat}{1}
\label{03071049}
u_{n}^{[1]}(T)
&
=
\sum_{p=1}^{n}
\RM(T)^{n-p}
\mathcal{L}
(T)^{p-1}
\notag
\\
&
=
n 
\mathcal{L}(T)^{n-1}
-
\sum_{k=2}^{n}
\mathcal{L}(T)^{n-k}
C(T^{k-1})
\notag
\\
&
=
\sum_{p=1}^{n}
\dbinom{n}{p}
\RM(T)^{n-p}
C(T)^{p-1}
\notag
\\
&
=
n
\mathcal{L}(T)^{n-1}
-
\sum_{s=0}^{n-2}
(n-s-1)
\mathcal{L}(T)^{n-(s+2)}
\RM(T)^{s}
C(T),
\end{alignat}
and
\begin{equation}
\label{12441607}
\|u_{n}^{[1]}(T)\|_{B(\A)}
\leq 
n\|T\|_{\A}^{n-1}.
\end{equation}
\end{lemma}
\proof
For 
brevity
in this proof we 
write
$\|\cdot\|$
for
$\|\cdot\|_{\A}$.
The cases $n=0,1$ are trivial.
Assume that $n\in\N-\{0,1\}$
and $T,h\in\A$
$$
\sum_{p=1}^{n}
\RM(T)^{n-p}
\mathcal{L}
(T)^{p-1}
(h)
=
hT^{n-1}
+
ThT^{n-2}
+
...
+
T^{k-1}hT^{n-k}
+
...
+
T^{n-1}h
$$
so
$$
(T+h)^{n}
=
T^{n}
+
\sum_{p=1}^{n}
\RM(T)^{n-p}
\mathcal{L}
(T)^{p-1}
(h)
+
\mathfrak{T}(h;T;2).
$$
Here
$
\mathfrak{T}(h;T;2)
$
is a polynomial
in the two
variables
$T$
and
$h$
each monomial
of which 
is at least 
of
degree 
$2$
with respect to the variable
$h$.
Hence
\begin{equation}
\label{18081442}
\begin{aligned}
\lim_{h\to\ze}
\frac{
\|
(T+h)^{n}
-
T^{n}
-
\sum_{p=1}^{n}
\RM(T)^{n-p}
\mathcal{L}
(T)^{p-1}
(h)
\|
}
{\|h\|}
&
=
\lim_{h\to\ze}
\frac{
\|
\mathfrak{T}(h;T;2)
\|
}
{\|h\|}
\\
&
\leq
\lim_{h\to\ze}
\frac{
\mathfrak{T}(\|h\|;\|T\|;2)
}
{\|h\|}
=0.
\end{aligned}
\end{equation}
Here
$\mathfrak{T}(\|h\|;\|T\|;2)$
is 
the 
polynomial
in the variables
$\|h\|$
and
$\|T\|$
obtained
by
replacing
in
$\mathfrak{T}(h;T;2)$
the variable
$h$
with
$\|h\|$
and
$T$
with
$\|T\|$.
Hence
\begin{equation}
\label{18081513}
u_{n}^{[1]}(T)
=
\sum_{p=1}^{n}
\RM(T)^{n-p}
\mathcal{L}
(T)^{p-1}
\end{equation}
and
\eqref{12441607}
and
the
first 
of
equalities 
\eqref{03071049}
follow.
Therefore
we have
for all
$T\in\A$
and
$h\in\A$
\begin{equation*}
\begin{aligned}
u_{n}^{[1]}(T)(h)
&
=
hT^{n-1}
+
ThT^{n-2}
+
...
+
T^{k-1}hT^{n-k}
+
...
+
T^{n-1}h
\\
&
=
hT^{n-1}
+
[T,h]T^{n-2}+hT^{n-1}
+
...
+
[T^{k-1},h]T^{n-k}+hT^{n-1}
+
...
\\
&
+
[T^{n-1},h]+hT^{n-1},
\end{aligned}
\end{equation*}
hence
\begin{equation}
\label{03071057}
u_{n}^{[1]}(T)(h)
=
nhT^{n-1}
+
\sum_{k=2}^{n}
[T^{k-1},h]T^{n-k}.
\end{equation}
This is 
the second
equality in
\eqref{03071049}.
The fourth equality
in \eqref{03071049}
follows
by the second one,
by the commutativity
property
in
\eqref{11071436}
and by
Corollary
\ref{06051405}.
By the first equality in
\eqref{03071049}
and
Lemma
\ref{01072007}
we obtain
the
third
equality in
\eqref{03071049}.
\endproof
\begin{definition}
\label{02102006}
Let $S$ be
a no empty set
and $X$ 
be
a Banach space over $\K$,
then we define 
\begin{equation}
\label{08071208}
\mc{B}(S,X)
\doteqdot
\left
\{
F:S\to X
\mid
\|
F
\|_{\mc{B}(S,X)}
\doteqdot
\sup_{u\in S}
\|F(u)\|_{X}
<\infty
\right
\}.
\end{equation}
Then
$
\lr{\mc{B}(S,X)}{\|\cdot\|_{\mc{B}(S,X)}}
$
is a Banach space over $\K$
and the convergence in it is called
the 
\emph{
uniform convergence 
on $S$ 
in 
$\|\cdot\|_{X}-$topology
},
or
simply
when this does not cause confusion,
the
\emph{
uniform convergence 
on $S$ 
},
(see Ch.$10$ of \cite{BourGT}).
\par\hspace{12pt}
Let
$
\{f_{\alpha}\}_{\alpha\in D}
\subset
\mc{B}(S,X)
$
then 
the
sum
$
\sum_{\alpha\in D}f_{\alpha}
$
\emph{
converges
uniformly
on
$S$
}
\footnote{
For the general 
definition
of a 
summable family
in a Hausdorff commutative
topological  group 
$G$
see
Ch. $3$, \S $5.1.$
of \cite{BourGT},
in our case 
$G=\mc{B}(S,X)$.
}
if
the net of all 
finite partial sums
converges
in 
$\mc{B}(S,X)$,
i.e.
by denoting
with
$\p_{\omega}(D)$
the direct ordered
set
of all finite subsets
of $D$ ordered
by inclusion,
there exists
$W\in\mc{B}(S,X)$
such that
\begin{equation}
\label{02101604sum}
\lim_{J\in\p_{\omega}(D)}
\sup_{u\in S}
\left
\|
W(u)
-
\sum_{\alpha\in J}
f_{\alpha}(u)
\right
\|_{X}
=
0.
\end{equation}
The 
sum
$
\sum_{\alpha\in D}
f_{\alpha}
$
\emph{
converges
absolutely
uniformly
on
$S$
}
or
\emph{
converges
absolutely
uniformly
for
$u\in S$
}
if
$$
\sum_{\alpha\in D}
\sup_{u\in S}
\|
f_{\alpha}(u)
\|_{X}
\doteq
\lim_{J\in\p_{\omega}(D)}
\sum_{\alpha\in J}
\sup_{u\in S}
\left
\|
f_{\alpha}(u)
\right
\|_{X}
<\infty.
$$
Since
$\mc{B}(S,X)$
is
a
Banach space,
the
absolute
uniform
convergence
implies
uniform
convergence.
Similar
definitions
for
sequences
and
in particular
for
series
$
\sum_{n=0}^{\infty}
f_{n}
$,
by replacing
$D$
with
$\N$,
while
$\sum_{\alpha\in J}$
and
$\lim_{\alpha\in\p_{\omega}(D)}$
with
resp.
$\sum_{n=0}^{N}$
and
$\lim_{N\to\infty}$,
finally
$\sum_{\alpha\in D}$
with
$\sum_{n=0}^{\infty}$.
\end{definition}
Now
we shall
show
that
a
power series 
$
g(T)
\doteqdot
\sum_{n=0}^{\infty}
\alpha_{n}
T^{n}
$
in a Banach algebra $\A$
is
Fr\'{e}chet differentiable
term by term,
the corresponding
power
series
of
its
Fr\'{e}chet
differential
$g^{[1]}$
is
uniformly convergent
on
$B_{r}(\ze)$
in the norm topology
of
$B(\A)$
for
all $0<r<R$,
and finally
that
$g^{[1]}$
is continuous,
where
the radius of convergence
$R$
of 
$
\sum_{n=0}^{\infty}
\alpha_{n}
\lambda^{n}
$
is different to zero.
The 
proofs are based on
the well-known results
stating that uniform convergence
in Banach spaces,
preserves
Fr\'{e}chet differentiability
and continuity,
see Theorem
$8.6.3.$ of the \cite{Dieud1}
for the first and
Theorem $(2)$, 
$\S 1.6.$, Ch. $10$
of the \cite{BourGT}
for the second one.
\par\hspace{12pt}
The 
Fr\'{e}chet
differentiability
of $g$
can be seen as a particular
case
of
the Fr\'{e}chet 
differentiability
of 
a power
series
of
polynomials
between two
Banach spaces
over $\mathbb{K}$,
whose
proof
for $\mathbb{K}=\C$,
was given
for the first time
in
\cite{martin};
while
the one 
for $\mathbb{K}=\R$,
given for the first time
in \cite{Michal}, 
used
a weak form of Markoff's 
inequality for the derivative 
of a polynomial,
see \cite{Schaeffer}.
\par
If $E_1$ and $E_2$ are Banach spaces, a homogeneous polynomial 
of degree $n$ on $E_1$ to $E_2$ is a function $p_n(x)$ with values in $E_2$, 
defined for all elements $x$ in $E_1$ and having the properties 
(a) $p_n(tx)=t^np_n(x)$, (b) $p_n(x+ty)$ is a polynomial of degree not greater 
than $n$ in the numerical variable $t$, with coefficients in $E_2$, 
(c) $\|p_n(x)\|\leq m\|x\|^n$ for some constant $m$ and every $x$; 
the smallest $m$ satisfying (c) is called the modulus $m(p_n)$ of the 
homogeneous polynomial. 
A series of the form $f(x)=\sum_0^\infty p_n(x)$ is called a power series. 
\par
Michal
in \cite{Michal}, 
considers real Banach spaces only, 
and defines the radius of analyticity $r$ of the power series as the radius 
of convergence of the ordinary power series $\sum_0^\infty m(p_n)t^n$. 
He proves that if $r>0$ the function $f(x)$ has 
Fr\'{e}chet 
differentials 
of all orders when $\|x\|<r$ and that these differentials 
are given by successive term-by-term differentiation of the series for $f(x)$. 
For complex Banach spaces this result is well known. 
It was first proved in
\cite{martin}
\par\hspace{12pt}
Our proof
in Lemma \ref{05051540}
has the advantage of
giving 
for the particular
case of
Banach algebras
a unified approach
for both the cases
real and complex.
\par
\begin{lemma}
[
Fr\'{e}chet 
differentiability of a 
power series in a Banach algebra 
]
\label{05051540}
Let $\A$ be a unitary Banach algebra, 
$\{\alpha_{n}\}_{n\in\N}\subset\K$
be
such that
the radius
of
convergence
of
the series 
$
g(\lambda)
\doteqdot
\sum_{n=0}^{\infty}
\alpha_{n}
\lambda^{n}
$
is
$R>0$.
\begin{enumerate}
\item
The 
series
$$
\sum_{n=0}^{\infty}
\alpha_{n}u_{n}
$$
converges
absolutely
uniformly
on
$B_{r}(\ze)$
for all
$0<r<R$.
\footnote{
By 
Def.
\ref{02102006},
$$
\sum_{n=0}^{\infty}
\sup_{T\in B_{r}(\ze)}
\|
\alpha_{n}
T^{n}
\|_{\A}
<
\infty
$$
for all $0<r<R$.
}
Hence
we can define
the map
$
g:B_{R}(\ze)\to\A
$
as
$
g(T)
\doteqdot
\sum_{n=0}^{\infty}
\alpha_{n}u_{n}(T)
$.
\label{A05051540}
\item
$
g
$
is 
Fr\'{e}chet differentiable
on
$B_{R}(\ze)$
and 
\begin{equation}
\label{02101544}
g^{[1]}
=
\sum_{n=1}^{\infty}
\alpha_{n}u_{n}^{[1]}.
\end{equation}
Here
the
series
converges 
absolutely
uniformly on
$B_{r}(\ze)$,
for all
$0<r<R$
\footnote{
By 
Def.
\ref{02102006}
$$
\sum_{n=1}^{\infty}
\sup_{T\in B_{r}(\ze)}
\|
\alpha_{n}
u_{n}^{[1]}(T)
\|_{B(\A)}
<
\infty
$$
for all 
$0<r<R$.
}
and
$g^{[1]}$
is continuous. 
\label{B05051540}
\end{enumerate}
\end{lemma}
\proof
For all
$r\in(0,R)$,
$T\in B_{r}(\ze)$
and
$n\in\N$
we have
$
\|\alpha_{n}T^{n}\|_{\A}
\leq
|\alpha_{n}|
\|T\|_{\A}^{n}
\leq
|\alpha_{n}|
r^n$,
so
$$
\sum_{n=0}^{\infty}
\sup_{T\in B_{r}(\ze)}
\|\alpha_{n}T^{n}\|_{\A}
\leq
\sum_{n=0}^{\infty}
|\alpha_{n}|r^{n}
<\infty.
$$
Which is statement $(1)$.
\par\hspace{12pt}
By
\eqref{12441607}
for all
$0<r<R$
$$
\sum_{n=0}^{\infty}
\sup_{T\in B_{r}(\ze)}
\|
\alpha_{n} u_{n}^{[1]}(T)
\|_{B(\A)}
\leq
\sum_{n=0}^{\infty}
|\alpha_{n}|n r^{n-1}
<
\infty.
$$
Hence
the series
$
\sum_{n=0}^{\infty}
\alpha_{n} u_{n}^{[1]}
$
converges 
absolutely
uniformly
on
$B_{r}(\ze)$
for all $0<r<R$.
Thus
the mapping
\begin{equation}
\label{08051437}
T
\ni
B_{R}(\ze)\subset\A
\mapsto
\sum_{n=0}^{\infty}
\alpha_{n} u_{n}^{[1]}(T)
\in
B(\A)
\end{equation}
is well defined 
on
$B_{R}(\ze)$
and
the series
converges
uniformly
for
$T\in B_{r}(\ze)$
for all $0<r<R$.
Hence
we can
apply 
Theorem
$8.6.3.$ of the \cite{Dieud1}
and then deduce 
\eqref{02101544}.
\par
\hspace{12pt}
Now 
it
remains to show the last 
part of the statement $(2)$,
i.e.
the continuity 
of the differential function 
$g^{[1]}$.
By the first part
of Lemma \ref{05051540a}
applied to the unitary
Banach algebra 
$B(\A)$
and by
\eqref{16051547}
for all
$n\in\N$
the maps
\begin{equation}
\label{10412403}
\A\ni T
\mapsto
\mathcal{L}(T)^{n}
\in
B(\A),
\,
\A\ni T
\mapsto
\RM(T)^{n}
\in
B(\A)
\end{equation}
and
the product on
$B(\A)\times B(\A)$ 
are continuous
in the norm 
topology
of
$B(\A)$,
so 
by
the first equality
in \eqref{03071049}
for all
$n\in\N$
\begin{equation}
\label{10051256}
u_{n}^{[1]}:\A\to B(\A)
\text{ 
is continuous.}
\end{equation}
By
\eqref{10051256},
the uniform 
convergence of which 
in
the
first part
of
statement $(2)$,
and 
finally by the fact that the 
set
of 
all
continuous maps
is closed 
with respect to the 
topology of uniform convergence, 
see for example 
Theorem $(2)$, 
$\S 1.6.$, Ch. $10$
of the \cite{BourGT},
we conclude that
for all
$0<r<R$
the mapping
$
g^{[1]}
\up
B_{r}(\ze):
B_{r}(\ze)
\subset
\A
\to
B(\A)
$
is continuous. 
This ends the proof 
of statement $(2)$.
\endproof
\begin{remark}
\label{14080915}
By 
statement $(2)$
of Lemma
\ref{05051540}
we have 
$$
g^{[1]}(T)(h)
=
\sum_{n=1}^{\infty}
\alpha_{n}u_{n}^{[1]}(T)(h).
$$
Here
the series
converges
absolutedly
uniformly 
for
$(T,h)\in B_{r}(\ze)\times B_{L}(\ze)$,
for all
$L>0$
and
$0<r<R$,
i.e.
$$
\sum_{n=1}^{\infty}
\sup_{(T,h)\in B_{r}(\ze)\times B_{L}(\ze)}
|\alpha_{n}|\|u_{n}^{[1]}(T)(h)\|_{\A}
<\infty
$$
\end{remark}
\begin{theorem}
[
\textbf{
Fr\'{e}chet
differential
of a power series
}
]
\label{05051540c}
Let $\A$ be a unitary Banach algebra, 
$\{\alpha_{n}\}_{n\in\N}\subset\K$
be
such that
the 
radius
of
convergence
of
the series
$
g(\lambda)
\doteqdot
\sum_{n=0}^{\infty}
\alpha_{n}
\lambda^{n}
$
is
$R>0$.
Then 
\begin{enumerate}
\item
for all $T\in B_{R}(\ze)$
\begin{equation}
\label{07071547}
g^{[1]}(T)
=
\sum_{n=1}^{\infty}
n\alpha_{n}\mathcal{L}(T)^{n-1}
-
\left\{
\sum_{p=0}^{\infty}
\left\{
\sum_{n=p+2}^{\infty}
(n-p-1)
\alpha_{n}
\mathcal{L}(T)^{n-(2+p)}
\right\}
\RM(T)^{p}
\right\}
C(T).
\end{equation}
Here
all the series
converge
absolutely
uniformly
on
$B_{r}(\ze)$
for all
$0<r<R$.
\label{E05051540c}
\item
for all $T\in B_{R}(\ze)$
\begin{equation}
\label{07071636}
g^{[1]}(T)
=
\sum_{n=1}^{\infty}
n\alpha_{n}\mathcal{L}(T)^{n-1}
-
\sum_{k=2}^{\infty}
\left\{
\sum_{n=k}^{\infty}
\alpha_{n}
\mathcal{L}(T)^{n-k}
\right\}
C(T^{k-1}).
\end{equation}
Here
all
the series
converge
absolutely
uniformly
on
$B_{r}(\ze)$
for all
$0<r<R$.
\label{F05051540c}
\item
For all
$T\in
B_{\frac{R}{3}}(\ze)$
\begin{equation}
\label{16081549}
g^{[1]}(T)
=
\sum_{p=1}^{\infty}
\frac{1}{p!}
(g)^{(p)}(\RM(T))
C(T)^{p-1}.
\end{equation}
Here
the
series
converges 
absolutely
uniformly
on
$
B_{r}(\ze)
$
for
all
$
0<r<\frac{R}{3}
$
and
$g^{(p)}:\K\to\K$
is
the
$p-$th derivative
of the function
$g$.
\label{C05051540c}
\end{enumerate}
\end{theorem}
\begin{remark}
\label{17081012}
If
$R/3\leq r<R$
then
in general the series in
\eqref{16081549}
may not converge, see for a 
counterexample
the
\cite{BurenkovDiff}.
\end{remark}
\begin{remark}
\label{17081112}
By using Def. \ref{03071132}
we have for all $T\in B_{R}(\ze)$
$$
\frac{1}{p!}
g^{(p)}(\RM(T))=
\sum_{n=p}^{\infty}
\dbinom{n}{p}\alpha_{n}\RM(T)^{n-p}\in B(\A).
$$
Clearly both 
\eqref{07071547}
and
\eqref{07071636}
immediately
imply
that
if
$T,h\in\A$
are such that
$[T,h]=\ze$,
then
$$
g^{[1]}(T)(h)
=
\sum_{n=1}^{\infty}
n\alpha_{n}hT^{n-1}.
$$
\end{remark}
\begin{proof}
[Proof of Theorem \ref{05051540c}]
By Lemmes \ref{05051540a}
and 
\ref{05051540}
\begin{alignat*}{1}
\label{15090938}
g^{[1]}(T)
&
=
\sum_{n=1}^{\infty}
\alpha_{n}u_{n}^{[1]}(T)
\\
&
=
\alpha_{1}\un
+
\sum_{n=2}^{\infty}
\alpha_{n}
\left(
n
\mathcal{L}(T)^{n-1}
-
\sum_{s=0}^{n-2}
(n-s-1)
\mathcal{L}(T)^{n-(s+2)}
\RM(T)^{s}
C(T)
\right).
\end{alignat*}
By \eqref{13221607}
for all $0<r<R$ 
$$
\sum_{n=2}^{\infty}
\sup_ {T\in B_{r}(\ze)}
\|
\alpha_{n}
n
\mathcal{L}(T)^{n-1}
\|_{B(\A)}
\leq
\sum_{n=2}^{\infty}
n|\alpha_{n}|r^{n-1}
<\infty
$$
and
\begin{gather}
\sum_{n=2}^{\infty}
\sum_{s=0}^{n-2}
\sup_{T\in B_{r}(\ze)}
\|
\alpha_{n}
(n-s-1)
\mathcal{L}(T)^{n-(s+2)}
\RM(T)^{s}
C(T)
\|_{B(\A)}
\leq
\notag
\\
\sum_{n=2}^{\infty}
|\alpha_{n}|
\sum_{s=0}^{n-2}
(n-s-1)
r^{n-2}
(2r)
=
\sum_{n=2}^{\infty}
|\alpha_{n}|
(n-1)nr^{n-1}
<\infty.
\label{15090956}
\end{gather}
\par
\hspace{12pt}
Therefore 
\begin{equation}
\label{02082601}
g^{[1]}(T)
=
\sum_{n=1}^{\infty}
\alpha_{n}
n
\mathcal{L}(T)^{n-1}
-
\sum_{n=2}^{\infty}
\alpha_{n}
\sum_{s=0}^{n-2}
(n-s-1)
\mathcal{L}(T)^{n-(s+2)}
\RM(T)^{s}
C(T).
\end{equation}
Here
each 
series 
converges
absolutely
uniformly
on
$B_{r}(\ze)$
for all $0<r<R$.
Inequality
\eqref{15090956}
also implies that
\begin{alignat*}{1}
&
\sum_{n=2}^{\infty}
\alpha_{n}
\sum_{s=0}^{n-2}
(n-s-1)
\mathcal{L}(T)^{n-(s+2)}
\RM(T)^{s}
C(T)
\\
&=
\sum_{s=0}^{\infty}
\sum_{n=s+2}^{\infty}
(n-s-1)
\alpha_{n}
\mathcal{L}(T)^{n-(2+s)}
\RM(T)^{s}
C(T)\\
&=
\left\{
\sum_{s=0}^{\infty}
\left\{
\sum_{n=s+2}^{\infty}
(n-s-1)
\alpha_{n}
\mathcal{L}(T)^{n-(2+s)}
\right\}
\RM(T)^{s}
\right\}
C(T).
\end{alignat*}
Here each series
converging
absolutely
uniformly
on
$B_{r}(\ze)$
for all $0<r<R$
and
statement $(1)$ follows.
Using \eqref{13221607}
we can estimate
\par
\hspace{12pt}
\begin{alignat}{2}
\label{New07071713}
&
\sum_{n=2}^{\infty}
\sum_{k=2}^{n}
\sup_{T\in B_{r}(\ze)}
\left
\|
\sum_{s=0}^{k-2}
\alpha_{n}
\mathcal{L}(T)^{n-k}
\RM(T)^{s}
\mathcal{L}(T)^{k-(2+s)}
C(T)
\right
\|_{B(\A)}
\notag\\
&
=
\sum_{n=2}^{\infty}
\sum_{k=2}^{n}
\sup_{T\in B_{r}(\ze)}
\left
\|
\sum_{s=0}^{k-2}
\alpha_{n}
\RM(T)^{s}
\mathcal{L}(T)^{n-(2+s)}
C(T)
\right
\|_{B(\A)}
\notag\\
&
\leq
\sum_{n=2}^{\infty}
\sum_{k=2}^{n}
\sum_{s=0}^{k-2}
|\alpha_{n}|
\sup_{T\in B_{r}(\ze)}
\left\|
\RM(T)
\right
\|_{B(\A)}^{s}
\left\|
\mathcal{L}(T)
\right
\|_{B(\A)}^{n-(2+s)}
\left\|
C(T)
\right
\|_{B(\A)}
\notag\\
&
\leq
2
\sum_{n=2}^{\infty}
\sum_{k=2}^{n}
\sum_{s=0}^{k-2}
|\alpha_{n}|
\sup_{T\in B_{r}(\ze)}
\left\|
T
\right
\|_{\A}^{n-1}
\notag\\
&
=
\sum_{n=2}^{\infty}
n(n-1)
|\alpha_{n}|
\sup_{T\in B_{r}(\ze)}
\left\|
T
\right
\|_{\A}^{n-1}
\notag\\
&
=
\sum_{n=2}^{\infty}
n(n-1)
|\alpha_{n}|
r^{n-1}
<\infty,
\end{alignat}
and therefore
\begin{alignat}{2}
\label{New07071728}
&
\sum_{n=2}^{\infty}
\sum_{k=2}^{n}
\sum_{s=0}^{k-2}
\alpha_{n}
\mathcal{L}(T)^{n-k}
\RM(T)^{s}
\mathcal{L}(T)^{k-(2+s)}
C(T)
\notag\\
&
=
\sum_{k=2}^{\infty}
\sum_{n=k}^{\infty}
\sum_{s=0}^{k-2}
\alpha_{n}
\mathcal{L}(T)^{n-k}
\RM(T)^{s}
\mathcal{L}(T)^{k-(2+s)}
C(T)
\notag\\
&
=
\sum_{k=2}^{\infty}
\sum_{n=k}^{\infty}
\alpha_{n}
\mathcal{L}(T)^{n-k}
\sum_{s=0}^{k-2}
\RM(T)^{s}
\mathcal{L}(T)^{k-(2+s)}
C(T)
\notag\\
&
=
\sum_{k=2}^{\infty}
\left\{
\sum_{n=k}^{\infty}
\alpha_{n}
\mathcal{L}(T)^{n-k}
\right\}
C(T^{k-1}).
\end{alignat}
\par
\hspace{12pt}
All 
the
series
uniformly converge
for
$
T\in B_{r}(\ze)
$.
Here in the last equality we used 
Corollary
\ref{06051405}
and
the 
fact that
$
\mathcal{L}(C(T^{k-1}))
\in
B(B(\A))
$.
Moreover
by 
\eqref{11071436}
\begin{equation*}
\sum_{n=2}^{\infty}
\sum_{k=2}^{n}
\sum_{s=0}^{k-2}
\alpha_{n}
\mathcal{L}(T)^{n-k}
\RM(T)^{s}
\mathcal{L}(T)^{k-(2+s)}
C(T)
=
\sum_{n=2}^{\infty}
\sum_{s=0}^{n-2}
(n-s-1)
\alpha_{n}
\mathcal{L}(T)^{n-(2+s)}
\RM(T)^{s}
C(T)
\end{equation*}
hence
by
\eqref{New07071728}
and
\eqref{02082601}
we obtain
statement 
$(2)$.
\par
\hspace{12pt}
Finally
we have for all
$r<\frac{R}{3}$
\begin{alignat*}{2}
\mathfrak{A}
&
\doteqdot
\sum_{n=1}^{\infty}
\sum_{p=1}^{n}
\sup_{T\in B_{r}(\ze)}
\left\|
\alpha_{n}
\dbinom{n}{p}
\RM(T)^{n-p}
C(T)^{p-1}
\right\|_{B(\A)}
\notag\\
&
\leq
\sum_{n=1}^{\infty}
\sum_{p=1}^{n}
\dbinom{n}{p}
|\alpha_{n}|
\sup_{T\in B_{r}(\ze)}
\|
\RM(T)
\|_{B(\A)}^{n-p}
\|
C(T)
\|_{B(\A)}^{p-1}
\notag\\
&
\leq
\sum_{n=1}^{\infty}
\sum_{p=1}^{n}
\dbinom{n}{p}
|\alpha_{n}|
\sup_{T\in B_{r}(\ze)}
\|
T
\|_{\A}^{n-p}
2^{p-1}
\|
T
\|_{\A}^{p-1}
\notag\\
&
=
\sum_{n=1}^{\infty}
\sum_{p=1}^{n}
\dbinom{n}{p}
|\alpha_{n}|
r^{n-1}
2^{p-1}.
\end{alignat*}
Hence
\begin{alignat*}{2}
\mathfrak{A}
&
\leq
\sum_{n=1}^{\infty}
|\alpha_{n}|
r^{n-1}
\sum_{p=1}^{n}
\dbinom{n}{p}
2^{p-1}
\notag\\
&
<
\sum_{n=1}^{\infty}
|\alpha_{n}|
r^{n-1}
\sum_{p=0}^{n}
\dbinom{n}{p}
2^{p}
\notag\\
&
=
\sum_{n=1}^{\infty}
|\alpha_{n}|
r^{n-1}
3^{n}
\notag\\
&
=
r^{-1}
\sum_{n=1}^{\infty}
|\alpha_{n}|
(3r)^{n}
<
\infty.
\end{alignat*}
\par
\hspace{12pt}
Thus
by 
the third 
equality in Lemma \ref{05051540a}
and
Lemma \ref{05051540}
we obtain statement $(3)$.
\end{proof}
The 
previous
Theorem 
\ref{05051540c}
is the 
main 
result
of the present work.
Let
$\A$
be
a unitary
$\K-$Banach algebra
and
$\sum_{n=0}^{\infty}\alpha_{n}\lambda^{n}$
a series at coefficients in $\K$
having
radius of convergence
$R>0$.
We
give
for the first
time the Fr\'{e}chet differential 
$g^{[1]}$
of the $\A-$valued
function
$g(T)=
\sum_{n=0}^{\infty}\alpha_{n}T^{n}$,
in 
a
\emph{
$C(T)-$
depending absolutely
uniformly convergent
series
on $B_{r}(\ze)$,
for all $0<r<R$,
in statement
\eqref{E05051540c};
and
in a
$C(T^{k})-$
depending absolutely
uniformly convergent
series
on $B_{r}(\ze)$,
for all $0<r<R$
and
with $k\geq 1$,
in statement
\eqref{F05051540c}}.
This 
allows
us
to give immediately 
a simplified 
formula 
for
the
value
$
g^{[1]}(T)(h)
$
in case of
the commutativity 
$[T,h]=\ze$,
with
$T\in B_{R}(\ze)$
and
$h\in\A$,
(see Remark \ref{17081112}).
\par\hspace{12pt}
Finally
we
give a different proof
respect
to 
\cite{Rudin1}
and in such a way generalizing
that
in 
\cite{BurenkovDiff},
of the  
known 
formula
in statement $(3)$,
in case
$0<r<\frac{R}{3}$,
see Remark \ref{18081019}
and Remark
\ref{08051500}.
\par
\begin{remark}
\label{15091737}
We note that the formula 
\eqref{07071547}
explicitly contains $C(T)$
as a factor, 
formula \eqref{07071636}
gives an expansion in terms
of $C(T^{k})$ 
and finally formula 
\eqref{16081549}
gives an expansion in terms
of $C(T)^{k}$.
\end{remark}
\begin{remark}
\label{18081019}
For all
$T$
such that
$
\|T\|<\frac{R}{3}
$
and for all
$h\in\A$
we have 
\begin{equation}
\label{D05051540c}
g^{[1]}(T)(h)
=
\sum_{p=1}^{\infty}
\frac{1}{p!}
g^{(p)}(T)
C(T)^{p-1}
(h).
\end{equation}
Here
the
series
is
uniformly convergent
for
$
(T,h)
\in
B_{r}(\ze)
\times
B_{L}(\ze)
$
for
all
$
0<r<\frac{R}{3}
$
and
$L>0$,
i.e.
$$
\sum_{p=1}^{\infty}
\frac{1}{p!}
\sup_{(T,h)\in B_{r}(\ze)\times B_{L}(\ze)}
\|g^{(p)}(T)
C(T)^{p-1}
(h)
\|_{\A}
<\infty.
$$
\end{remark}
\begin{corollary}
[Fr\'{e}chet 
differential
of a power series
of differentiable
functions
defined on
an open set 
of
a $\K-$Banach Space
and at values
in a $\K-$Banach algebra $\A$]
\label{08051458}
Let $\A$ be a unitary Banach algebra, 
and
$\{\alpha_{n}\}_{n\in\N}\subset\K$
be
such that
the
radius of convergence
of
the series 
$
g(\lambda)
\doteqdot
\sum_{n=0}^{\infty}
\alpha_{n}
\lambda^{n}
$
is
$R>0$
and
$0<r<R$.
Finally
let 
$X$ 
be a Banach space over $\K$,
$D\subseteq X$ an open set in $X$
and 
$\mc{T}:D\to\A$ 
a Fr\'{e}chet differentiable mapping
such
that
$
\mc{T}(D)\subseteq B_{r}(\ze)
$
or alternatively 
$
D
$
is convex
and
bounded
and
$
\sup_{x\in D}\|\mc{T}^{[1]}(x)\|_{B(X,\A)}
<\infty
$.
If we set
$
\w{r}
\doteqdot
\sup_{x\in D}
\|
\mc{T}(x)
\|_{\A}
$,
then
\begin{enumerate}
\item
$
\w{r}
<
\infty
$
and if
$\w{r}<R$
then
$$
g\circ\mc{T}
=
\sum_{n=0}^{\infty}
\alpha_{n}\mc{T}^{n}.
$$
Here
the series
is
uniformly convergent
on
$D$,
while
$\mc{T}^{n}:
D\ni x
\mapsto
\mc{T}(x)^{n}$.
\label{A08051458}
\item
If
$0<\w{r}<R$
then
the function 
$
g\circ\mc{T}
$
is 
Fr\'{e}chet differentiable and
\begin{equation}
\label{08051520}
\left[g\circ\mc{T}\right]^{[1]}(x)
=
\sum_{n=0}^{\infty}
\alpha_{n}
u_{n}^{[1]}(\mc{T}(x))
\mc{T}^{[1]}(x),
\quad
\forall x\in D.
\end{equation}
Here
the series
converges
in
$B(X,\A)$.
\label{B08051458}
Moreover
\begin{enumerate}
\item
If
$\mc{T}^{[1]}:D\to B(X,\A)$
is
continuous 
then
the function
$
\left[g\circ\mc{T}\right]^{[1]}:
D
\to
B(X,\A)
$,
is also
continuous.
\label{1B08051458}
\item
If
$\sup_{x\in D}\|\mc{T}^{[1]}(x)\|_{B(X,\A)}
<\infty$,
then
the series in 
\eqref{08051520}
absolutely
uniformly converges
on
$D$.
\label{2B08051458}
\end{enumerate}
\end{enumerate}
\end{corollary}
\proof
Let
us
consider
the case
in which
$D$ 
is
convex and bounded,
and
$
\sup_{x\in D}
\|
\mc{T}^{[1]}(x)
\|_{B(X,\A)}
<\infty$.
Let
$a,b\in D$ 
and 
$S_{a,b}$ 
the segment jointing $a,b$.
$D$ is
convex so
$S_{a,b}\subset D$.
By an application of the Mean Value Theorem, 
see 
Theorem $8.6.2.$ of \cite{Dieud1}, 
we have for any 
$x_{0}\in D$
$$
\|
\mc{T}(b)-\mc{T}(a)-\mc{T}^{[1]}(x_{0})(b-a)
\|_{\A}
\leq
\|b-a\|_{X}
\cdot
\sup_{x\in S_{a,b}}
\|
\mc{T}^{[1]}(x)
-
\mc{T}^{[1]}(x_{0})
\|_{B(X,\A)}.
$$
Thus
by 
$
\|A\|-\|B\|
\leq
\|A-B\|
$
in any normed space,
we have
fixed $b\in D$ and $x_{0}\in D$,
that for all
$a\in D$
\begin{alignat}{1}
\label{08061825}
&
\sup_{a\in D}
\|
\mc{T}(a)
\|_{\A}
\\
&
\leq
\sup_{a\in D}
\|
\mc{T}(b)-\mc{T}^{[1]}(x_{0})(b-a)
\|_{\A}
+
\sup_{a\in D}
\|b-a\|_{X}
\cdot
\sup_{a\in D}
\sup_{x\in S_{a,b}}
\|
\mc{T}^{[1]}(x)
-
\mc{T}^{[1]}(x_{0})
\|_{B(X,\A)}
\notag
\\
&
\leq
\|
\mc{T}(b)
\|_{\A}
+
\|
\mc{T}^{[1]}(x_{0})
\|_{B(X,\A)}
\sup_{a\in D}
\|
b-a
\|_{X}
+
\sup_{a\in D}
\|b-a\|_{X}
\cdot
\sup_{x\in D}
\|
\mc{T}^{[1]}(x)
-
\mc{T}^{[1]}(x_{0})
\|_{B(X,\A)}
\notag
\\
&
\leq
\|
\mc{T}(b)
\|_{\A}
+
\sup_{a\in D}
\|b-a\|_{X}
\cdot
\left(
2
\|
\mc{T}^{[1]}(x_{0})
\|_{B(X,\A)}
+
\sup_{x\in D}
\|
\mc{T}^{[1]}(x)
\|_{B(X,\A)}
\right)
<
\infty.
\notag
\end{alignat}
Here $D$
is 
considered 
bounded 
and
$
\sup_{x\in D}
\|
\mc{T}^{[1]}(x)
\|_{B(X,\A)}
<\infty
$
by hypothesis.
So by 
\eqref{08061825}
$\w{r}<\infty$
which is
the
first 
part of statement 
\eqref{A08051458}.
\par
\hspace{12pt}
Let
$D\subseteq X$
be
the
open set 
of which 
in the hypotheses.
By
$\w{r}<\infty$
we can assume
that
$0<\w{r}<R$,
then
the second part of
statement 
\eqref{A08051458}
follows
by 
statement
\eqref{A05051540}
of
Lemma
\ref{05051540}.
\par
\hspace{12pt}
In the sequel
of the proof we assume
that
$0<\w{r}<R$.
By 
statement 
\eqref{B05051540}
of  
Lemma
\ref{05051540} 
and
by
the Chain Theorem, 
see 
$8.2.1.$ of the \cite{Dieud1},
$
g
\circ
\mc{T}
$
is Fr\'{e}chet differentiable and
its differential map is
\begin{equation}
\label{08051647}
\left[g\circ\mc{T}\right]^{[1]}:
D
\ni
x
\mapsto
g^{[1]}(\mc{T}(x))
\circ
\mc{T}^{[1]}(x)
=
\left
\{
\sum_{n=0}^{\infty}
\alpha_{n}
u_{n}^{[1]}(\mc{T}(x))
\right
\}
\circ
\mc{T}^{[1]}(x)
\in
B(X,\A),
\end{equation}
here
it was used the fact that
uniform
convergence
implies
puntual convergence.
By 
statement 
\eqref{B05051540}
of  
Lemma
\ref{05051540} 
and 
$\w{r}<R$
the previous series 
converges
in $B(\A)$,
moreover
$
\mf{b}:
B(\A)
\times
B(X,\A)
\ni
(\phi,\psi)
\mapsto
\phi\circ\psi
\in
B(X,\A)$
is a
bilinear 
and
continuous
map
i.e.
\begin{equation}
\label{08071600}
\mf{b}
\in
B_{2}(B(\A)\times B(X,\A);B(X,\A)),
\end{equation}
since
$ 
\|\phi\circ\psi\|_{B(X,\A)}
\leq
\|\phi\|_{B(\A)}
\cdot
\|\psi\|_{B(X,\A)}
$.
Thus
\eqref{08051647}
implies
\eqref{08051520}.
\par
\hspace{12pt}
Set 
$$
\Gamma:D\ni
x
\mapsto
(
g^{[1]}\circ\mc{T}(x),
\mc{T}^{[1]}(x)
)
\in
B(\A)\times B(X,\A).
$$
According to
\eqref{08051647}
\begin{equation}
\label{10051428}
\left[g\circ\mc{T}\right]^{[1]}
=
\mf{b}
\circ
\Gamma.
\end{equation}
$\mc{T}^{[1]}$ is continuous by hypothesis,
and
$
g^{[1]}
$
is continuous
by statement 
\eqref{B05051540}
of Lemma
\ref{05051540},
while
$\mc{T}$ is continuous 
being differentiable by hypothesis,
so
$g^{[1]}\circ\mc{T}$
is continuous.
Therefore
by
Proposition
$1$, $\S 4.1.$, Ch $1$, of the 
\cite{BourGT}
the map
$\Gamma$ is continuous.
Thus
by
\eqref{10051428} 
and
\eqref{08071600}
$\left[g\circ\mc{T}\right]^{[1]}$
is continuous
and
statement
\eqref{1B08051458}
follows.
\par
\hspace{12pt}
By
\eqref{12441607}
\begin{alignat*}{1}
&
\sum_{n=0}^{\infty}
\sup_{x\in D}
\|
\alpha_{n}
u_{n}^{[1]}(\mc{T}(x))
\circ
\mc{T}^{[1]}(x)
\|_{B(X,\A)}
\\
&
\leq
\sum_{n=0}^{\infty}
\sup_{x\in D}
|\alpha_{n}|
\|
u_{n}^{[1]}(\mc{T}(x))
\|_{B(\A)}
\cdot
\|
\mc{T}^{[1]}(x)
\|_{B(X,\A)}
\\
&
\leq
\sum_{n=0}^{\infty}
\sup_{x\in D}
|\alpha_{n}|
n
\|
\mc{T}(x)
\|_{\A}^{n-1}
\cdot
\|
\mc{T}^{[1]}(x)
\|_{B(X,\A)}
\\
&
\leq
\sum_{n=0}^{\infty}
\sup_{x\in D}
|\alpha_{n}|
n
\|
\mc{T}(x)
\|_{\A}
^{n-1}
\cdot
\sup_{x\in D}
\|
\mc{T}^{[1]}(x)
\|_{B(X,\A)}
\\
&
\leq
M
\sum_{n=0}^{\infty}
|\alpha_{n}|
n
\w{r}^{n-1}
<
\infty,
\end{alignat*}
where
$
M
\doteqdot
\sup_{x\in D}
\|
\mc{T}^{[1]}(x)
\|_{B(X,\A)}
$
and
statement
\eqref{2B08051458}
follows.
\endproof
\begin{remark}
\label{G08051458}
By
\eqref{08051647},
statement
$(3)$ 
of Theorem \ref{05051540c}
and
\eqref{08071600},
if
$
0<\w{r}<\frac{R}{3}
$,
we have for all
$x\in D$ 
\begin{equation}
\label{03071955}
\left[
g\circ\mc{T}
\right]^{[1]}(x)
=
\sum_{p=1}^{\infty}
\frac{1}{p!}
g^{(p)}(\RM(\mc{T}(x)))
C(\mc{T}(x))^{p-1}
\mc{T}^{[1]}(x).
\end{equation}
In addition if
$\sup_{x\in D}\|\mc{T}^{[1]}(x)\|_{B(X,\A)}
<\infty$,
then
the series in the 
\eqref{03071955}
is 
absolutely
uniformly convergent
on
$D$.
If
$
0<\w{r}<\frac{R}{3}
$
by
\eqref{03071955}
we have for all
$h\in X$
\begin{equation}
\label{06071746}
\left[
g\circ\mc{T}
\right]^{[1]}
(x)(h)
=
\sum_{p=1}^{\infty}
\frac{1}{p!}
g^{(p)}(\mc{T}(x))
C(\mc{T}(x))^{p-1}
(\mc{T}^{[1]}(x)(h)).
\end{equation}
In addition
if
$
\sup_{x\in D}\|\mc{T}^{[1]}(x)\|_{B(X,\A)}
<\infty
$,
then
the series in the 
\eqref{06071746}
is 
absolutely
uniformly convergent
for
$(x,h)\in D\times B_{L}(\ze)$,
for all
$L>0$,
i.e.
$$
\sum_{p=1}^{\infty}
\frac{1}{p!}
\sup_{(x,h)\in D\times B_{L}(\ze)}
\|
g^{(p)}(\mc{T}(x))
C(\mc{T}(x))^{p-1}
(\mc{T}^{[1]}(x)(h))
\|_{\A}<\infty.
$$
\end{remark}
\begin{remark}
\label{14081324}
In a similar way 
as in
the proof
of statement 
\ref{2B08051458}
of Corollary
\ref{08051458}
we have
\begin{enumerate}
\item
if
$
0<\w{r}<R
$
\begin{equation}
\label{07072013II}
\begin{aligned}
&
\left[g\circ\mc{T}\right]^{[1]}
(x)(h)
=
\sum_{n=1}^{\infty}
n\alpha_{n}
\mc{T}^{[1]}(x)(h)
\mc{T}(x)^{n-1}+
\\
&
\sum_{n=2}^{\infty}
\sum_{p=0}^{n-2}
(n-p-1)
\alpha_{n}
\mc{T}(x)^{p}
[\mc{T}(x),\mc{T}^{[1]}(x)(h)]
\mc{T}(x)^{n-(2+p)}
\\
&
=
\sum_{n=1}^{\infty}
n\alpha_{n}
\mc{T}^{[1]}(x)(h)
\mc{T}(x)^{n-1}
+
\\
&
\sum_{p=0}^{\infty}
\sum_{n=p+2}^{\infty}
(n-p-1)
\alpha_{n}
\mc{T}(x)^{p}
[\mc{T}(x),\mc{T}^{[1]}(x)(h)]
\mc{T}(x)^{n-(2+p)}.
\end{aligned}
\end{equation}
If in addition
$\sup_{x\in D}\|\mc{T}^{[1]}(x)\|_{B(X,\A)}
<\infty$,
then 
all
the
series
in
\eqref{07072013II}
are
absolutely
uniformly convergent 
for 
$(x,h)\in D\times B_{L}(\ze)$,
for all
$L>0$.
\label{H08051458}
\item
If
$0<\w{r}<R$
we have
\begin{equation}
\label{07071636II}
\begin{aligned}
&
\left[g\circ\mc{T}\right]^{[1]}
(x)
\\
&
=
\sum_{n=1}^{\infty}
n\alpha_{n}
\mathcal{L}(\mc{T}(x))^{n-1}
\mc{T}^{[1]}(x)
-
\sum_{k=2}^{\infty}
\left\{
\sum_{n=k}^{\infty}
\alpha_{n}
\mathcal{L}(\mc{T}(x))^{n-k}
\right\}
C(\mc{T}(x)^{k-1})
\mc{T}^{[1]}(x).
\end{aligned}
\end{equation}
If in addition 
$\sup_{x\in D}\|\mc{T}^{[1]}(x)\|_{B(X,\A)}<\infty$,
then 
all
the
series
in 
\eqref{07071636II}
are
absolutely
uniformly convergent 
on
$D$.
\label{I08051458}
\end{enumerate}
\end{remark}
\begin{definition}
\label{29081058}
Let 
$\lr{G}{\|\cdot\|_{G}}$ 
be 
a $\C-$Banach space, 
then 
we denote 
by
$G_{\R}$
the vector space 
$G$
over $\R$
whose 
operation of 
summation
is the same 
of that of the 
$\C-$vector space
$G$, 
and whose 
multiplication
by
scalars 
is the 
restriction 
to
$\R\times G$
of
the
multiplication
by
scalars 
on
$\C\times G$,
finally we set
$
\|\cdot\|_{G_{\R}}
\doteqdot
\|\cdot\|_{G}
$.
Then
$\lr{G_{\R}}{\|\cdot\|_{G_{\R}}}$ 
is
a
Banach space
over $\R$
and will be called the 
$\R-$Banach space associated to 
$\lr{G}{\|\cdot\|_{G}}$.
\par\hspace{12pt}
Let
$F,G$ 
be two $\C-$Banach spaces
then
of course
$
B(F,G)
\subset
B(F_{\R},G_{\R})
$,
where the inclusion is to be intended only 
as
a set inclusion.
Let
$A\subseteq F$
then
if
$A$
is open
in
$F$
it is
open
also in
$F_{\R}$.
For a mapping 
$
f:A\subseteq F\to G
$,
we will denote with the symbol 
$
f^{\R}:A\subseteq F_{\R}\to G_{\R}
$
the same mapping but considered defined 
in the 
subset 
$A$
of the
$\R-$Banach space
associated to 
$F$
and at values
in the 
$\R-$Banach space
associated to 
$G$.
\end{definition}
\begin{remark}
\label{08051500}
Let $Y,Z$ be two
$\C-$Banach spaces,
then by considering that 
$B(Y,Z)\subset B(Y_{\R},Z_{\R})$,
we have that 
for
each Fr\'{e}chet differential function
$f:A\subseteq Y\to Z$
the same function
$f^{\R}:A\subseteq Y_{\R}\to Z_{\R}$
considered in the corresponding 
real Banach spaces,
is differentiable, 
in addition
$
f^{[1]}
=
(f^{\R})^{[1]}
$.
Therefore
if we get a real Banach space 
$X$,
we shall obtain 
the
same
statements 
of 
Corollary
\ref{08051458},
Remark
\ref{G08051458}
and
Remark \ref{14081324}
by replacing
$\A$
with
$\A_{\R}$.
In particular 
take
$X\doteqdot\R$,
and recall 
that 
for every differential 
map
$H:D\subseteq\R\to\A_{\R}$
we have
$H^{[1]}(t)(1)
=
\frac{d\,H}{d\,t}(t)$
for all $t\in D$,
where
$\frac{d\,H}{d\,t}:D\to\A_{\R}$
is the derivative of $H$.
Hence
if we denote
$
g(\lambda)
\doteqdot
\sum_{n=0}^{\infty}
\alpha_{n}
\lambda^{n}
$
and
assume 
$0<\w{r}<\frac{R}{3}$,
then
we obtain
by
\eqref{06071746}
that
for all
$t\in D\subseteq\R$ 
\begin{equation}
\label{04071034}
\frac{
d\,
g^{\R}\circ\mc{T}
}{dt}
(t)
=
\sum_{p=1}^{\infty}
\frac{1}{p!}
g^{(p)}(\mc{T}(t))
C(\mc{T}(t))^{p-1}
\left(
\frac{d\,\mc{T}}{dt}(t)
\right). 
\end{equation}
In addition
if
$
\sup_{t\in D}
\|
\frac{d\,\mc{T}}{dt}(t) 
\|_{\A}
<\infty
$,
then
the series in the 
\eqref{04071034}
is 
absolutely
uniformly convergent
on
$D$.
This formula
has been 
shown
for the first time
by Victor I. Burenkov
in
\cite{BurenkovDiff}.
\end{remark}
Notice that 
$C(\mc{T}(t))^{0}=\un$
and for all
$n\in\N-\{0\}$
$$
C(\mc{T}(t))^{n}
\left(
\frac{d\,\mc{T}}{dt}(t)
\right) 
=
\overbrace{
\biggl[
\cdots
\biggl[
\biggl[
}^{n}
\frac{d\,\mc{T}}{dt}(t),
\mc{T}(t)
\overbrace{
\biggr],
\mc{T}(t)
\biggr],
\cdots
\biggr]}^{n}.
$$
In particular if 
$\left[\frac{d\,\mc{T}}{dt}(t),\mc{T}(t)\right]=\ze$,
then
\begin{equation}
\label{16210903}
\frac{
d\,
g^{\R}
\circ
\mc{T}
}{dt}
(t)
=
g^{(1)}(\mc{T}(t))
\frac{d\,\mc{T}}{dt}(t).
\end{equation}
If
$
\left[
\left[
\frac{d\,\mc{T}}{dt}(t),
\mc{T}(t)
\right],
\mc{T}(t)
\right]
=\ze
$
then
\begin{equation}
\label{05180801}
\frac{
d\,
g^{\R}
\circ
\mc{T}
}{dt}
(t)
=
g^{(1)}(\mc{T}(t))
\frac{d\,\mc{T}}{dt}(t)
+
\frac{1}{2}
g^{(2)}(\mc{T}(t))
\left[
\frac{d\,\mc{T}}{dt}(t),
\mc{T}(t)
\right]
\end{equation}
and so on.
\begin{corollary}
\label{08051809}
Let $\A$ be a unitary Banach algebra, 
$\{\alpha_{n}\}_{n\in\N}\subset\K$
be
such that
the
radius
of
convergence
of
the series
$
g(\lambda)
\doteqdot
\sum_{n=0}^{\infty}
\alpha_{n}
\lambda^{n}
$
is
$R>0$.
Finally
let
$
W\in\A
$,
$0<r<R$,
$
D_{(r,W)}
\doteqdot
\left
]
-
\frac{r}{\|W\|},
\frac{r}{\|W\|}
\right
[
$
with the convention
$\frac{r}{0}\doteqdot\infty$
and
$$
\mc{T}(t)
=
t
W
$$
for all
$
t
\in
D_{(r,W)}
$.
Then
with the notations adopted 
in the statements of 
Lemma
\ref{05051540},
we have 
\begin{enumerate}
\item
$$
g^{\R}
\circ 
\mc{T}(t)
=
\sum_{n=0}^{\infty}
\alpha_{n}t^{n}W^{n}
$$
and the series
is
absolutely
uniformly convergent
for
$
t
\in
D_{(r,W)}
$.
\item
$
g^{\R}
\circ 
\mc{T}
$
is 
derivable,
the following map 
\begin{equation}
\label{08051520e}
\frac{
d
\,g^{\R}
\circ 
\mc{T}
}{dt}(t)
=
W
\sum_{n=1}^{\infty}
\alpha_{n}
n 
t^{n-1}
W^{n-1}
=
W
\frac{
d\,g
}{d\lambda}
\circ\mc{T}(t)
\end{equation}
for all
$
t\in
D_{(r,W)}
$
is the
derivative function 
of 
$g^{\R}\circ\mc{T}$,
is continuous 
and
the series in the 
\eqref{08051520e}
is absolutely
uniformly convergent.
\end{enumerate}
\end{corollary}
\proof
Statement $(1)$ is 
trivial.
The map $\mc{T}$ is derivable 
with constant derivative equal to $W\in\A$,
hence
we have statement $(2)$
by 
Remark
\ref{08051500}
and
\eqref{07072013II}.
\endproof
\section{
Application to the 
analytic functional calculus in 
a
$\C-$Banach space
}
\label{10052127}
In 
this section
$G$ is
a complex
Banach space
and
$Open(\C)$
is the set of all 
open subsets of $\C$.
We denote
by
$\sigma(T)$
the spectrum
of
$T$
for all
$T\in B(G)$,
and for all
$U\in Open(\C)$
such that
$\sigma(T)\subset U$
and
$g:U\to\C$
analytic,
by
$g(T)$
the operator belonging to $B(G)$
as defined 
in the 
analytic functional calculus 
framework 
given in 
Definition 
$7.3.9.$ of the \cite{ds},
that is
$$
g(T)
\doteqdot
\frac{1}{2\pi i}
\int_{B}g(\lambda)R(\lambda;T)d\,\lambda.
$$
Here 
$R(\lambda;T)\doteqdot(\lambda\un-T)^{-1}$
is the resolvent of $T$,
while
$B\subset U$ 
is the boundary of an open set containing $\sigma(T)$
and consisting of a 
finite number of rectifiable Jordan curves.
If
$U$
is an open 
neighborhood of
$0$
and 
$g(\lambda)
=
\sum_{n=0}^{\infty}
\alpha_{n}\lambda^{n}$
for all 
$\lambda\in U$,
then
by
Theorem
$7.3.10.$ of the \cite{ds}
$
g(T)
=
\sum_{n=0}^{\infty}
\alpha_{n}T^{n}
$
converging
in $B(G)$.
Therefore
for this case
we can apply
all
the 
results 
in 
Section
\ref{15061005}.
\begin{corollary}
[
Fr\'{e}chet 
differential
of an 
operator valued
analytic function
defined on
an open set 
of
a $\R-$Banach Space
]
\label{09051552}
Let 
$U_{0}$
be
an open 
neighborhood of
$0\in\C$,
$g:U_{0}\to\C$
an analytic function
such that
$g(\lambda)
=
\sum_{n=0}^{\infty}
\alpha_{n}\lambda^{n}$,
for all $\lambda\in U_{0}$.
Let
$R>0$
be
the 
radius
of
convergence
of
the series 
$
\sum_{n=0}^{\infty}
\alpha_{n}
\lambda^{n}
$.
Finally
let 
$X$ 
be a Banach space over $\R$,
$D\subseteq X$ an open set of $X$
and $\mc{T}:D\to B(G)_{\R}$ 
a Fr\'{e}chet differentiable mapping
such that there exists
$r\in\R^{+}\mid 0<r<R$ 
such that
\begin{enumerate}
\item
$\mc{T}(D)\subseteq B_{r}(\ze)$
\item
$\sigma(\mc{T}(x))\subseteq U_{0}$,
for all $x\in D$.
\end{enumerate}
Then
\begin{enumerate}
\item
$$
g^{\R}\circ\mc{T}
=
\sum_{n=0}^{\infty}
\alpha_{n}\mc{T}^{n}.
$$
Here
the series
absolutely
uniformly converges
on
$D$.
\label{A09051552}
\item
The
statements
of
Corollary
\ref{08051458},
Remark
\ref{G08051458}
and
Remark
\ref{14081324}
hold with
$\A$
replaced by
$B(G)_{\R}$,
while
Remark
\ref{08051500}
holds with
$\A$
replaced by
$B(G)$.
\label{B09051552}
\end{enumerate}
\end{corollary}
\proof
The map
$g^{\R}\circ\mc{T}$
is well defined 
by 
the condition
$\sigma(\mc{T}(x))\subseteq U_{0}$
for all 
$x\in D$,
while
the power series expansion 
follows
by
Theorem
$7.3.10.$ of the \cite{ds}.
Therefore
statement 
\ref{A09051552}
follows
by
the hypothesis 
$\mc{T}(D)\subseteq B_{r}(\ze)$,
with
$0<r<R$
and
Remark
\ref{08051500}.
Statement  
\ref{B09051552}
is by
Corollary
\ref{08051458}
and
Remark
\ref{08051500}.
\endproof
\begin{remark}
\label{05071034}
If we assume that
$G$
is a complex Hilbert space
and
$\mc{T}(x)$
is a normal operator
for all 
$x\in D$,
then
the condition
$\mc{T}(D)\subseteq B_{r}(\ze)$
is equivalent
to the following one
$\sigma(\mc{T}(x))\subseteq B_{r}(\ze)$
for all 
$x\in D$.
\end{remark}
Although
the following is a well-known
result, for the sake 
of completeness
we shall give a proof 
by using 
Corollary
\ref{08051809}.
\begin{corollary}
\label{09051748}
Let $\{\alpha_{n}\}_{n\in\N}\subset\C$
be 
such that the radius of convergence of
the series 
$
g(\lambda)
=
\sum_{n=0}^{\infty}\alpha_{n}\lambda^{n}
$
is $R>0$.
Moreover
let
$W\in B(G)$,
$0<r<R$,
and
$
D_{(r,W)}
\doteqdot
\left
]
-
\frac{r}{\|W\|},
\frac{r}{\|W\|}
\right
[
$
with the convention
$\frac{r}{0}\doteqdot\infty$.
Then the operator
$g(tW)$ is well defined
for all
$t\in
D_{(r,W)}
$
and
$$
g(tW)
=
\sum_{n=0}^{\infty}
\alpha_{n}t^{n}W^{n}.
$$
Here the series converges absolutely uniformly
on 
$
D_{(r,W)}
$.
Moreover the map
$
D_{(r,W)}
\ni 
t
\mapsto
\frac{dg}{d\lambda}(tW)
$
is Lebesgue integrable in $B(G)$ 
in the sense
defined in
\cite{IntBourb},
Definition
$2$
Ch. $IV$, \S $3$, $n^{\circ} 4$,
and
for all
$
u_{1},u_{2}\in
D_{(r,W)}
$
$$
W
\int_{u_{1}}^{u_{2}}
\frac{dg}{d\lambda}(tW)
\,
dt
=
g(u_{2}W)
-
g(u_{1}W).
$$
Here
$
\int_{u_{1}}^{u_{2}}
\frac{dg}{d\lambda}(tW)
\,
dt
$
is the
Lebesgue integral
of the map
$
D_{(r,W)}
\ni 
t
\mapsto
\frac{dg}{d\lambda}(tW)
$
as defined
in
Definition
$1$
Ch. $IV$, \S $4$, $n^{\circ} 1$
of
\cite{IntBourb}.
\end{corollary}
\proof
By \eqref{10501634}
\begin{equation}
\label{09052028}
\RM(W)\in B(B(G)).
\end{equation}
Set
$
\mc{T}(t)\doteqdot tW
$
for all
$
t\in
D_{(r,W)}
$.
Then
$
\frac{d\,g}{d\,\lambda}
\circ
\mc{T}(t)
=
\sum_{n=1}^{\infty}
\alpha_{n}
n
t^{n-1}
W^{n-1}
$
and 
the
map
$
D_{(r,W)}
\ni
t
\mapsto
\frac{
d\,g
}{d\lambda}
\circ\mc{T}(t)
$
is continuous
in
$B(G)$,
as a corollary
of 
\eqref{08051520e},
by replacing the map
$g$ 
with
$
\frac{d g}{d\lambda}
$,
hence
it
is 
Lebesgue-measurable
in
$B(G)$.
Finally
let
$
u_{1},u_{2}
\in
D_{(r,W)}
$
\begin{alignat*}{1}
\int_{[u_{1},u_{2}]}^{*}
\|
\frac{d\,g}{d\,\lambda}
(tW)
\|
d
\,
t
&
=
\int_{[u_{1},u_{2}]}^{*}
\|
\sum_{n=1}^{\infty}
\alpha_{n}
n 
t^{n-1}
W^{n-1}
\|
d
\,
t\\
&
\leq
\int_{[u_{1},u_{2}]}^{*}
\sum_{n=1}^{\infty}
|\alpha_{n}|n r^{n-1} d\,t\\
&
=
|u_{2}-u_{1}|
\sum_{n=1}^{\infty}
|\alpha_{n}|
n 
r^{n-1}
<
\infty.
\end{alignat*}
Here
$
\int_{[u_{1},u_{2}]}^{*}
$
is the upper integral
of the Lebesgue measure
on 
$[u_{1},u_{2}]$.
By
this 
boundedness
and
by its 
Lebesgue-
measurability
we conclude 
by
Theorem $5$,
$IV.71$
of
\cite{IntBourb} 
that
$
[u_{1},u_{2}]
\ni
t
\mapsto
\frac{
d\,g
}{d\lambda}
\circ\mc{T}(t)
\in 
B(G)
$
is Lebesgue-integrable 
in
$B(G)$,
so in particular
by Definition
$1$,
$IV.33$
of
\cite{IntBourb}
\begin{equation}
\label{09052027}
\exists
\int_{u_{1}}^{u_{2}}
\frac{d\,g}{d\,\lambda}
\circ\mc{T}(t)d\,t
\in B(G).
\end{equation}
Therefore 
by
\eqref{09052028},
\eqref{09052027},
Theorem $1$,
$IV.35$
of
\cite{IntBourb} 
and
\eqref{08051520e}
\begin{equation}
\label{21061013}
W
\int_{u_{1}}^{u_{2}}
\frac{d\,g}{d\,\lambda}
\circ\mc{T}(t)
d\,t
=
\int_{u_{1}}^{u_{2}}
W
\frac{d\,g}{d\,\lambda}
\circ\mc{T}(t)
d\,t
=
\int_{u_{1}}^{u_{2}}
\frac{d\,g^{\R}(\mc{T}(t))}{dt}
d\,t.
\end{equation}
\par
Furthermore
by 
the continuity
of
the map
$
D_{(r,W)}
\ni
t
\mapsto
\frac{
d\,g
}{d\lambda}
\circ\mc{T}(t)
$
in
$B(G)$
and
by
\eqref{08051520e},
$
D_{(r,W)}
\ni
t
\mapsto
\frac{d\,g^{\R}(\mc{T}(t))}{dt}
$
is
continuous
in
$B(G)$
and
it
is
the derivative of
the map
$
D_{(r,W)}
\ni
t
\mapsto
\,g^{\R}
\circ 
\mc{T}(t)
$.
Therefore
it
is 
Lebesgue integrable
in
$B(G)$,
where 
the integral
has to be understood
as defined in
Ch
$II$
of
\cite{FRV},
see 
Proposition
$3$,
$n^{\circ}3$,
$\S 1$,
Ch
$II$
of
\cite{FRV}.
Finally
the 
Lebesgue integral
for functions with values
in a Banach space
as defined in 
Ch
$II$
of
\cite{FRV},
turns
to be
the 
integral
with respect to the Lebesgue
measure
as defined in
Ch. $IV$, \S $4$, $n^{\circ} 1$ 
of
\cite{IntBourb}
(see 
Ch $III$, \S $3$, $n^{\circ} 3$
and 
example in 
Ch $IV$, \S $4$, $n^{\circ} 4$ 
of 
\cite{IntBourb}).
Thus
the
statement 
follows
by
\eqref{21061013}.
\endproof
Finally we want to remark
that
one of the main aims of 
our work \cite{SilInt}
is 
proving this formula 
for a certain class 
of unbounded operators in a Banach space
and by considering the integral
in weaker topologies than that 
induced by the
norm 
in $B(G)$.



\bibliography{references}
\bibliographystyle{amsplain}

\end{document}